\newif\ifcomments\commentsfalse
\documentclass[a4paper,11pt]{article}
\usepackage{array}
\usepackage{theorem}
\usepackage{amsmath,amscd,amssymb}
\usepackage{latexsym}
\usepackage{epsfig}
\usepackage{tikz}
\usetikzlibrary{intersections}
\theorembodyfont{\sl}

\newtheorem{lemma}{Lemma}[section]

\newtheorem{proposition}[lemma]{Proposition}
\newtheorem{theorem}[lemma]{Theorem}
\newtheorem{corollary}[lemma]{Corollary}
\newtheorem{conjecture}[lemma]{Conjecture}
\newtheorem{example}[lemma]{Example}
\newtheorem{definition}[lemma]{Definition}
\newtheorem{convention}[lemma]{Convention}

\newtheorem{question}[lemma]{Question}

\newcommand{\BB}{\mathbb B}

\newcommand{\NN}{\mathbb N}

\newcommand{\RR}{\mathbb R}
\renewcommand{\SS}{\mathbb S}

\newcommand{\ZZ}{\mathbb Z}

\newcommand{\cE}{\mathcal E}
\newcommand{\cF}{\mathcal F}

\newcommand{\cP}{\mathcal P}

\renewcommand{\Bar}{\overline}

\newcommand{\imic}{\cong}

\newcommand{\inj}{\hookrightarrow}

\renewcommand{\Tilde}{\widetilde}

\newcommand{\pr}{\mathop{\mathrm {pr}}\nolimits}

\newcommand{\bE}{\mathbf E}
\newcommand{\bF}{\mathbf F}

\newcommand{\Int}{\mathop{\mathrm {Int}}\nolimits}

\newenvironment{proof}{\paragraph{Proof:}}{\hfill$\square$\smallskip}
\begin{document}
\title{Regular cylindrical algebraic decomposition}
\author{J.H.~Davenport, A.F.~Locatelli \& G.K.~Sankaran\\
Department of Mathematical Sciences,\\ University of Bath,\\ Bath
BA2~7AY, UK
\\
{\tt J.H.Davenport@bath.ac.uk}\\ {\tt
  acyr.locatelli86@gmail.com}\\ {\tt G.K.Sankaran@bath.ac.uk}
}
\maketitle

Abstract: We show that a strong well-based cylindrical algebraic
decomposition $\cP$ of a bounded semi-algebraic set is a regular cell
decomposition, in any dimension and independently of the method by
which $\cP$ is constructed. Being well-based is a global condition on
$\cP$ that holds for the output of many widely used algorithms. We
also show the same for $S$ of dimension $\le 3$ and $\cP$ a strong
cylindrical algebraic decomposition that is locally boundary simply
connected: this is a purely local extra condition.
\smallskip

\section{Introduction}\label{sect:introduction}
Cylindrical algebraic decomposition (abbreviated {c.a.d.}; see
Definition~\ref{def:cad}) is a method that decomposes a semi-algebraic
subset $S\subseteq\RR^n$ into simpler pieces (cells) in a systematic way. It
first arose~\cite{collins} in the context of quantifier elimination,
but has since become a useful technique for effective computation of
topological invariants, such as homology groups, of semi-algebraic
sets. For example, the piano movers' problem (see, inter alia,
\cite{schwartz1983}) asks whether the configuration space of
allowable positions of an object in a subset of $\RR^3$ is
connected. Questions of this nature can have both theoretical and
practical importance.
\footnote{Some of the results of this paper formed part of the Bath
  Ph.D.\ thesis~\cite{locatelli} of the second author, which was
  funded by the University of Bath. We acknowledge discussions with
  Matthew England and David Wilson, and EPSRC grant EP/J003247/1 which
  funded them.  GKS thanks Andrew Ranicki and Kenichi Ohshika for
  education about cobordism.}

Cell decompositions can be quite pathological, however, and for
purposes of computation (again, both theoretical and practical) some
further conditions are needed. One would hope, at least, to obtain a
representation of $S$ as a CW-complex: better still, a regular cell
complex.

\begin{question}\label{qu:main}
Let $S \subset \RR^n$ be a closed and bounded semi-algebraic set.
\begin{enumerate}
\item[(i)] Can we find a c.a.d.\ of $S$ into regular cells?
\item[(ii)] Given a c.a.d.\ of $S$, can we tell easily whether it is a
  regular cell decomposition?
\end{enumerate}
\end{question}

A partial answer to Question~\ref{qu:main}(i) was given in
\cite{basu2014}, where it is shown that the bounded cells
of a semi-monotone c.a.d.\ (see~\cite{basu2013}) are regular, and an
algorithm is given to construct such a c.a.d.\ if $\dim S\le 2$ or
$n=3$. However, being semi-monotone is a strong condition and it is
not at present clear whether semi-monotone c.a.d.{s} exist at all in
general. Even if they do, they are likely to be laborious to construct
and to have many cells, making them unsuitable for computation.

We can always find a c.a.d.\ of $S$ with regular cells if we allow a
change of coordinates, but this is usually undesirable.  From a
computational point of view, implementations of c.a.d.\ algorithms
often improve run time by exploiting sparseness, which is destroyed by
change of coordinates. Quantifier elimination, the original motivating
example for c.a.d.\ in~\cite{collins}, does not allow arbitrary
changes of coordinates, and indeed requires some ordering on the
coordinates: we assume, as is usual in c.a.d.\ theory, a total order
$x_1\prec x_2\prec\cdots\prec x_n$. Thus it is important to understand
which c.a.d.s have good properties such as giving regular cell
decompositions.

For these reasons the earlier study of the topological properties of
c.a.d.s in \cite{lazard2010} remains very relevant to
Question~\ref{qu:main}. Lazard describes some much weaker conditions
and conjectures (Conjecture~\ref{conj:strongcadisreg}, below) that a
c.a.d.\ satisfying them will have regular cells, and also shows how to
construct these c.a.d.s for $n=3$.

Much earlier, Schwartz and Sharir~\cite{schwartz1983} had proved that
a c.a.d.\ produced by Collins' algorithm~\cite{collins}, the only
method known at that time, gives a regular cell complex provided it is
well-based (see Definition~\ref{def:wellbased}).

In Section~\ref{subsect:Finvt} of this paper we prove that any
well-based strong {c.a.d}\ gives a regular cell complex: see
Theorem~\ref{thm:wellbasedstronggivesreg} for the precise statement
and Definition~\ref{def:strongcad} for the meaning of ``strong''. A
well-based c.a.d.\ produced by Collins' algorithm is always strong, so
this is a generalisation of the result of Schwartz and Sharir, but it
is entirely independent of the method used to construct the c.a.d. and
is thus more widely applicable.

In Section~\ref{subsect:topology}, we prove a slightly weaker form of
Lazard's conjecture for $\dim S=2$ or $n=3$: see
Theorem~\ref{thm:vstrongcadisreg}. Our methods also suggest a strategy
for $n\ge 4$.

These two results are superficially similar but quite different in
detail. In Section~\ref{subsect:Finvt} we consider a c.a.d.\ $\cP$
that is $\bF$-invariant (see Definition~\ref{def:Finvariantcad}) for a
large set of polynomials $\bF\subset \RR[x_1,\ldots,x_n]$, including
as a minimum all the polynomials that are used to define~$S$. (Indeed,
the term well-based itself already presumes that $\cP$ is
$\bF$-invariant.) In one respect, this restriction is not onerous:
algorithms commonly do produce $\bF$-invariant c.a.d.s by
construction. On the one hand, $\bF$-invariance is a very strong
global condition, which may force $\cP$ to have many cells even far
from~$S$ and cannot be checked locally near each cell.

By contrast, in Section~\ref{subsect:topology} we are concerned with
the topology of c.a.d.s in general, subject only to local conditions.
Apart from its theoretical interest, this is potentially important in
the context of Brown's NuCAD algorithm \cite{Brown2015}, which
constructs cells that are capable of being cells in {c.a.d.}s, rather
than complete {c.a.d.}s, and is thus inherently local: global
conditions such as $\bF$-invariance do not arise.

This difference is also reflected in the methods of proof of
Theorem~\ref{thm:wellbasedstronggivesreg} and
Theorem~\ref{thm:vstrongcadisreg}. For
Theorem~\ref{thm:wellbasedstronggivesreg}, we use largely elementary
methods of real algebraic geometry, exploiting the rigidity imposed by
the $\bF$-invariance. The tools used to prove
Theorem~\ref{thm:vstrongcadisreg} are topological and are anything but
elementary as they include the $h$-cobordism theorem (in effect, the
Poincar\'e Conjecture).

Some of the statements make sense over an arbitrary real closed field,
but as our methods are in part topological we work over $\RR$
throughout. See~\cite{demdah} for an approach to $h$-cobordism in the
context of real closed fields, which could possibly allow one to
remove this restriction.

A subsidiary aim of this paper is to give some consistent terminology
for ideas that have appeared in different parts of the literature
under various, sometimes incompatible, names. We try to do this in the
course of Section~\ref{sect:celldecompositions}, which explains the
background to the problems. The main results are found in
Section~\ref{sect:cad}. Finally, in Section~\ref{sect:ordercomplex},
we make some brief observations about another question
(Question~\ref{qu:lazardrevised}) raised by Lazard in the same
paper~\cite{lazard2010}.

\section{Cells and cell decompositions}\label{sect:celldecompositions}

Throughout the rest of the paper, we use $\BB(p,\varepsilon)$ and
$\SS(p,\varepsilon)$ to denote the open ball and the sphere,
respectively, in $\RR^n$ with centre $p$ and radius $\varepsilon$ (the
dimension will always be clear): we use $\BB^n$ and $\SS^{n-1}$ for
the standard unit ball and sphere in $\RR^n$. If $X\subset \RR^n$ then
$\Bar X$ denotes the closure of $X$ in the Euclidean topology.

We begin with a well-known example, which motivates
Question~\ref{qu:main}.

\begin{example}\label{ex:whitney}
Put $\Delta=\{(x,y) \in \RR^2 \mid 0 < x < 1,\, -x < y < x\}$ and
consider the semi-algebraic set
\[
W=\{(x,y,z) \in \Delta\times \RR \subset \RR^3 \mid x^2z=y^2 \},
\]
a subset of the Whitney umbrella $\{(x,y,z) \in \RR^3 \mid  x^2z-y^2=0 \}$.
\begin{center}
  \begin{tikzpicture}[scale=2]

\draw[dotted](0,-1)--(0,1.5);
\draw[dotted](-1,0)--(1.5,0);
\draw[blue, fill=green](0,0)--(0,1)--(0.7,0.8) to [out=290,
  in=180](1,0) -- cycle;
\draw[green](0,0)--(1,0);
\draw[blue, fill=green!50](0.7,0.8) to [out=270,
  in=180](1,0) to [out=0, in=270](1.3,1.2) -- (0,1) -- cycle;
\draw[red, fill = red] (0,0) circle [radius = 0.025];
\draw[red, fill = red] (0,1) circle [radius = 0.025];
\draw[red, fill = red] (0.7,0.8) circle [radius = 0.025];
\draw[red, fill = red] (1.3,1.2) circle [radius = 0.025];
\draw[dotted](-0.3,-0.5)--(0.9,1.5);
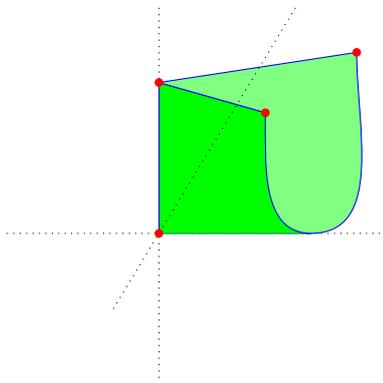
\node [below] at (0,-1.3){\rm Example~\ref{ex:whitney}: a CW complex that
    is not regular};
\end{tikzpicture}
\end{center}
We can partition $\Bar{W}$ into nine disjoint cells: the corners
$(0,0,0)$, $(0,0,1)$, $(1,1,1)$ and $(1,-1,1)$; the edges
$\{(0,0,t)\}$, $\{(t,t,1)\}$ and $\{(t,-t,1)\}$ (for $0<t<1$) together
with $\{(1,t,t^2)\}$ (for $-1<t<1$); and $W$.
\end{example}

This cell decomposition is a c.a.d.\ of $\Bar{W}$ and represents
$\Bar{W}$ as a CW complex~\cite[pp.~5\ \&\ 519]{hatcher2002} but not
as a regular cell complex (see Definition~\ref{def:regularcell}).

\subsection{Definitions and terminology}\label{subsect:terminology}

Next, we collect some definitions. This is already not quite trivial,
because the same or very similar conditions have been introduced by
several authors at different times under very different, and sometimes
incompatible, names. Before doing any mathematics at all, we propose
some terminology which we believe is consistent, flexible and
memorable.

\begin{definition}\label{def:cell}
A subset $C$ of $\RR^n$ is a \emph{$d$-cell}, for $d\in\NN_0$, if
there exists a homeomorphism $\BB^d \to C$, for some $d \in \NN_0$
called the \emph{dimension} $\dim C$ of $C$.  The \emph{boundary of
$C\subset\RR^n$} (sometimes, for emphasis, the \emph{cell boundary})
is $\partial_c C=\Bar{C} \setminus C$.
\end{definition}

The cell boundary of a cell $C$ does not, in general, coincide with
the topological boundary of $C$, which is $\Bar{C} \setminus
\Int(C)$. Also, $\Bar{C}$ may have a structure of manifold with
boundary in which the manifold boundary $\partial_m\Bar{C}$ (see
Definition~\ref{def:mfldwbdy}) might coincide with neither $\partial_c
C$ nor the topological boundary.

\begin{definition}\label{def:celldecomposition}
Let $X$ be a subset of $\RR^n$. A \emph{cell decomposition of $X$} is
a partition $\cP=\{C_\alpha\}$ of $X$ into disjoint cells.
\end{definition}

Even in $\RR^2$ a cell may have bad boundary: for instance, if we take
$C=\{(x,y)\in\RR^2\mid y\cos x<1\}$ then $\partial_c C$ has infinitely
many connected components.  We define below some desirable conditions
on a cell $C$ and its boundary. Some of these conditions are intrinsic
to $C$; others are related to a cell decomposition.

Recall that if $X\subseteq X'$ and $Y\subseteq Y'$ are inclusions of
topological spaces, a homeomorphism $\varphi\colon (X',X)\to (Y',Y)$
is a homeomorphism $\varphi\colon X'\to Y'$ such that $\varphi|_X$ is
a homeomorphism $X\to Y$.

\begin{definition}\label{def:regularcell}
We say that two subsets $X\subset \RR^n$ and $Y\subset \RR^m$ are
\emph{equiregular} if there exists a homeomorphism $\varphi\colon
(\Bar{X},X)\to (\Bar{Y},Y)$. A $d$-cell $C\subset\RR^n$ is said to be
a \emph{regular cell} if $C$ is equiregular with $\BB^d$.
\end{definition}

The hypercube $(0,1)^d\subset \RR^n$ is a regular cell. On the other
hand, the cell $\BB^2 \setminus \{(0,y) \in \RR^2 \mid y \geq 0\}$ is
not regular, even though its closure is $\Bar{\BB}^2$.  Moreover, even
if $C$ is regular, a particular homeomorphism $\BB\to C$ need not
extend to $\Bar{\BB}\to \Bar{C}$ even as a continuous map.

We establish a convention for naming cell decompositions where all
cells have a certain property.

\begin{convention}\label{conv:decompositionnames}
If $\Pi$ is a property of cells we shall say that $\cP$ is a $\Pi$
cell decomposition if $\cP$ is a cell decomposition and all cells of
$\cP$ satisfy~$\Pi$.
\end{convention}
Thus a regular cell decomposition is a decomposition into regular
cells. Not every property of $\cP$ can be checked on the cells,
however: for example, a finite cell decomposition is simply a
decomposition into finitely many cells. There is no ambiguity, because
finiteness is not a property of cells.

\begin{definition}\label{def:adjacent}
We say that two cells $C$ and $D$ in $\RR^n$ are \emph{adjacent} if
either one intersects the closure of the other. We say that $C$
is \emph{subadjacent} to $D$, written $C\preceq D$, if
$C \cap \Bar{D}\neq \emptyset$.
\end{definition}

If $D\preceq C$ and $D\cap C=\emptyset$, then by
\cite[Theorem~5.42]{basuroypollack}, $\dim D \leq \dim C - 1$.

Now we define some extrinsic properties of a cell, in relation to a
cell decomposition.

\begin{definition}\label{def:closurefinite}
Let $C$ be a cell of a cell decomposition $\cP$. We say that $C$ is
\emph{closure finite in $\cP$} if $\Bar{C}$ (or, equivalently,
$\partial_c C$) is the union of finitely many cells of $\cP$.
\end{definition}

This condition is found in the literature under different names. It is
called \emph{boundary coherent} in \cite[Definition~2.7]{lazard2010}:
elsewhere, sometimes without the finiteness requirement, it is called
the \emph{frontier condition}.  We use the term {\it closure finite}
as it is more descriptive than either of the terms above, and is the
usual term in the topology literature. Indeed, it is the meaning of
the ``C'' of ``CW complex'': see~\cite[p.~520]{hatcher2002}.

\subsection{Examples and basic properties}\label{subsect:examples}

We first illustrate some relations among the properties introduced in
Section~\ref{subsect:terminology}.

\begin{example}\label{ex:noncf}
Consider the following cell decomposition of $[0,1]^3$, in which the
end points $(0,0,1/2)$ and $(0,1,1/2)$ are not cells, and the edges
$(0,0,z)$ and $(0,1,z)$ ($z\in(0,1)$) are not sub-divided by these
points.
\begin{center}
\begin{tikzpicture}[scale=0.5]
\draw[blue](0,0) -- (0,3) -- (2,4.5) -- (4.7,4.5) -- (3,3) -- (3,0) --
cycle;
\draw[blue](0,3) -- (3,3);
\draw[blue](4.7,4.5) -- (4.7,2) -- (3,0);
\draw[dotted, blue](0,0) -- (2,2) -- (2,4.5);
\draw[dotted, blue](2,2) -- (4.7,2);
\draw[green](0,1.5)--(3,1.5);
\draw[red, fill=red](0,0) circle [radius = 0.04];
\draw[red, fill=red](0,3) circle [radius = 0.04];
\draw[red, fill=red](2,4.5) circle [radius = 0.04];
\draw[red, fill=red](4.7,4.5) circle [radius = 0.04];
\draw[red, fill=red](3,3) circle [radius = 0.04];
\draw[red, fill=red](3,0) circle [radius = 0.04];
\draw[red, fill=red](2,2) circle [radius = 0.04];
\draw[red, fill=red](4.7,2) circle [radius = 0.04];
\node[below] at (0,-0.3){\rm Example~\ref{ex:noncf}: subadjacency and closure finiteness I};
\end{tikzpicture}
\end{center}
The cube is closure finite; the $1$-cell
$[0,1]\times\{0\}\times \{\frac{1}{2}\}$, subadjacent to the cube, is
not closure finite.
\end{example}

Example~\ref{ex:noncf} also shows that even if $C$ is closure finite
in $\cP$ its boundary may contain a cell of $\cP$ that is not closure
finite.

\begin{example}\label{ex:cfsubadj}
The cells $C_1 = \SS^1 \setminus \{(x,y) \in \RR^2 \mid x \ge 0 \mbox{
and } y \le 0 \}$ and $C_1 = \{(x,0) \in \RR^2 \mid -1 < x < 2\}$ in
$\RR^2$ are not closure finite in $\cP=\{C_1,C_2\}$, and they are
subadjacent to each other.
\end{example}
\begin{center}
\begin{tikzpicture}
\draw[blue](-1,0)--(2,0);
\draw[red](1,0)arc(0:270:1);
\node [above] at (-1,0.5){$C_1$};
\node [below] at (1.5,0){$C_2$};
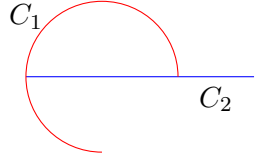
\node [below] at (0,-1.3){\rm Example~\ref{ex:cfsubadj}: subadjacency and closure finiteness II};
\end{tikzpicture}
\end{center}
\begin{lemma}\label{lem:cfbdysubadj}
Let $C$ be a cell of a cell decomposition $\cP$. Then $C$ is closure
finite if and only if $D \preceq C$ implies that
$D \subseteq \Bar{C}$. In particular, if two cells $C$ and $D$ of
$\cP$ are closure finite and subadjacent to each other, then $C = D$.
\end{lemma}

\begin{proof}
Suppose that $C$ is closure finite in $\cP$; that is, $\partial_c C =
\bigcup_{i=1}^{k}D_i$. If $D \preceq C$ then it must intersect some
$D_i$, and thus $D = D_i$. In particular $D \subset \partial_c C$.
Conversely, if $\Bar{C}$ contains all cells subadjacent to $C$, then
$\Bar{C} = \bigcup_{D\preceq C} D$.
\end{proof}

\begin{definition}\label{def:wellbordered}
Let $C$ be a cell of a cell decomposition $\cP$. We say that $C$ is
\emph{well-bordered in $\cP$} if there is a finite collection
$\{C_i\}\subset\cP$ of cells of $\cP$ such that $\dim C_i = \dim C -
1$ and $\partial_c C = \bigcup_i \Bar{C_i}$.
\end{definition}

For instance a $2$-sphere minus a point is a cell that is not
well-bordered. See \cite[Example~2.9]{lazard2010} for more examples of
how a cell can fail to be well-bordered, and for a closure finite
decomposition that is not well-bordered.

Like closure finiteness, the well-bordered property does not
permeate to subadjacent cells.

\begin{example}\label{ex:wbnotcf}
Consider the cell decomposition that consists of the open cube
$(0,1)^3$, all six of its faces, eleven of its edges (not the
$z$-axis) and seven of its corners (not the origin), together with the
$1$-cell $\{0\}\times\{0\}\times(-1,1)$ and the $0$-cell
$(0,0,-1)$. We observe the following:
\begin{enumerate}
\item The cell $(0,1)^3$ is well-bordered but not closure finite.
\item the faces that are adjacent to the cell
$\{0\}\times\{0\}\times(-1,1)$ are neither well-bordered nor closure
finite.
\end{enumerate}
\end{example}
\begin{center}
\begin{tikzpicture}[scale=0.5]
\draw[blue](0,-3) -- (0,3) -- (2,4.5) -- (4.7,4.5) -- (3,3) -- (3,0) --
(0,0);
\draw[blue](0,3) -- (3,3);
\draw[blue](4.7,4.5) -- (4.7,2) -- (3,0);
\draw[dotted, blue](0,0) -- (2,2) -- (2,4.5);
\draw[dotted, blue](2,2) -- (4.7,2);
\draw[red, fill=red](0,3) circle [radius = 0.04];
\draw[red, fill=red](2,4.5) circle [radius = 0.04];
\draw[red, fill=red](4.7,4.5) circle [radius = 0.04];
\draw[red, fill=red](3,3) circle [radius = 0.04];
\draw[red, fill=red](3,0) circle [radius = 0.04];
\draw[red, fill=red](2,2) circle [radius = 0.04];
\draw[red, fill=red](4.7,2) circle [radius = 0.04];
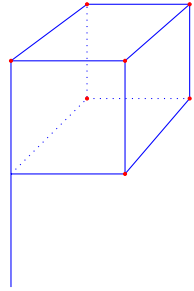
\node[below] at (0,-3.5){\rm Example~\ref{ex:wbnotcf}: well-borderedness and closure finiteness};
\end{tikzpicture}
\end{center}

These two conditions are nevertheless related.

\begin{lemma}\label{lem:wbclosurecf}
Let $C$ be a cell of a cell decomposition $\cP$. If $C$ and all cells
subadjacent to $C$ are well-bordered, then $C$ is closure finite.
\end{lemma}

\begin{proof}
In view of Lemma~\ref{lem:cfbdysubadj} it suffices to show that
$D\subseteq\partial_c C$ if $D\preceq C$ and $D \neq C$. We proceed by
induction on $\dim C$: for $\dim C=0$ there is nothing to prove.

As $C$ is well-bordered, $\partial_c C = \bigcup_i \Bar{C_i}$ for some
finite collection of cells $C_i$ with $\dim C_i = \dim C -1$. If
$D \cap C_i \neq \emptyset$, for some $i$, then $D = C_i$; otherwise,
$D \cap \Bar{C_i} \neq \emptyset$ for some $i$. Then by induction
$D \subset \partial_c C_i \subset \partial_c C$ and the result
follows.
\end{proof}

\begin{corollary}\label{cor:wbimpliescf}
Any well-bordered cell decomposition is closure finite.
\end{corollary}

As we have seen, \cite[Example~2.9]{lazard2010} shows that the
converse is not true.  However, for regular cell decompositions the
two coincide.

\begin{lemma}\label{lem:regthencf=wb}
A regular cell decomposition of a compact set $S \subset \RR^n$ is
closure finite if and only if it is well-bordered.
\end{lemma}

\begin{proof}
If $C$ is a closure finite regular $d$-cell then $\partial_c C$ is
homeomorphic to $\SS^{d-1}$, and decomposes into finitely many
cells. Then $\partial_c C$ is the closure of the union of the
$(d-1)$-cells in that decomposition.

The other direction is just Corollary~\ref{cor:wbimpliescf}.
\end{proof}

\begin{definition}\label{def:locallyboundary}
For any topological property $\Pi$, we say that a set $C\subset \RR^n$
is \emph{locally boundary $\Pi$} if every $p\in \partial_c C$ has a
base of neighbourhoods $N$ in $\Bar C$ such that $\Pi$ holds for each
$N\cap C$.
\end{definition}

In many cases one may take the neighbourhoods $N$ to be the
intersections $\BB(p,\varepsilon)\cap C$ for $0<\varepsilon \ll 1$: we
shall do this without further comment when it is convenient, but one
should check that it is permissible to do so. An example of a property
$\Pi$ for which this would not be permissible is disconnectedness:
$\RR_+\subset\RR$ is locally boundary disconnected (as well as being
locally boundary connected!) because we may take for $N$ the sets
$[0,\frac{1}{r})\cup(\frac{2}{r},\frac{3}{r})$, but the intervals
$(0,\varepsilon)$ are all connected. We shall not in fact consider any
properties for which the balls are not suitable neighbourhoods.

Lazard \cite[Definition~2.7]{lazard2010} defines ``boundary smooth'',
which according to Definition~\ref{def:locallyboundary} is the same as
``locally boundary connected''. We prefer this terminology because it
extends to other properties (we shall need ``locally boundary simply
connected'' later, for instance) and because the term ``smooth'' is
already overloaded. In particular, ``boundary smooth'' has nothing to
do with either being $C^\infty$ or the absence of singularities.

\subsection{Semi-algebraic cell decompositions}\label{subsect:semialgdecomp}

Now we limit ourselves to semi-algebraic cells. We shall make constant
use of the conic structure of semi-algebraic sets
\cite[Theorem~9.3.6]{BCR}. We also need a slightly stronger relative
version (take $X=\emptyset$ to recover the usual version).

\begin{proposition}\label{relativeconic}
Suppose that $X\subseteq Y$ are semialgebraic subsets of $\RR^n$ and
$p\in Y$. Then for $0<\varepsilon\ll 1$ there exists a semi-algebraic
homeomorphism
$\psi\colon \Bar\BB(p,\varepsilon)\to \Bar\BB(p,\varepsilon)$ which is
the identity on $\SS(p,\varepsilon)$, such that
$\|\psi(q)-p\|=\|q-p\|$ for all $q\in \Bar\BB(p,\varepsilon)$ and
$\psi\big((Y,X)\cap \Bar\BB(p,\varepsilon)\big)$ is the cone on
$(Y,X)\cap \Bar\SS(p,\varepsilon)$ with vertex $p$.
\end{proposition}
\begin{proof}
Consider
\[
M:=\{(y,t)\in \RR^n\times \RR\mid (y\in Y\text{ and
  }t=0)\text{ or }(y\in X\text{ and }0\le t \le 1)\}.
\]
This is a semi-algebraic set (it is the mapping cylinder of the
inclusion $X\inj Y$) so we may consider its conic structure near a
point $p\in M$ where $t=0$. Then it is sufficient to take $\psi$ to be
the restriction to $t=0$ of the map
$\phi\colon \Bar\BB(p,\varepsilon)\to \Bar\BB(p,\varepsilon)$ in
$\RR^{n+1}$ guaranteed by \cite[Theorem~9.3.6]{BCR}.
\end{proof}

The following consequence of the local conic structure is also
useful.

\begin{proposition}\label{prop:BretractsS}
Let $C\subset\RR^n$ be a semi-algebraic set and $p \in \partial_c
C$. Then for $0<\varepsilon\ll 1$ the intersection $C\cap
\BB(p,\varepsilon)$ has $C\cap \SS(p,\varepsilon)$ as a deformation
retract.
\end{proposition}
\begin{proof}
Applying \cite[Theorem~9.3.6]{BCR} to $C\cup\{p\}$ yields a
homeomorphism between $\Bar{\BB}(p,\varepsilon) \cap C\cup\{p\}$ and
the cone on $\SS(p,\varepsilon)\cap C$. Away from $\{p\}$, this
restricts to a homeomorphism between $(\Bar{\BB}(p,\varepsilon) \cap
C)\setminus\{p\}$ and $(\SS(p,\varepsilon)\cap C) \times (0,1]$, and
the latter retracts onto $(\SS(p,\varepsilon)\cap
C) \times \{\frac{1}{2}\}$.
\end{proof}

Because of Proposition~\ref{prop:BretractsS} we can often replace the
ball with a sphere when checking Definition~\ref{def:locallyboundary}.

\begin{corollary}\label{cor:homotopylb}
If $\Pi$ is a homotopy property for which
Definition~\ref{def:locallyboundary} can be checked on balls, and $C
\subset \RR^n$ is a semi-algebraic cell, then $C$ is locally boundary
$\Pi$ if and only if, for all $p\in \partial_c C$, there exists
$\delta>0$ such that $C\cap \SS(p,\varepsilon)$ has property $\Pi$ for
all $0<\varepsilon<\delta$.
\end{corollary}
Clearly the same is also true with $\Bar\BB$ instead of $\SS$.

With the definitions we have made, being locally boundary $\Pi$ is
automatically an equiregularity invariant property. In particular, as
was pointed out in \cite{schwartz1983}, a regular cell, even if not
semi-algebraic, is always locally boundary connected.

\section{Cylindrical algebraic decomposition}\label{sect:cad}

We think of a cylindrical algebraic decomposition as a finite
partition of $\RR^n$ into semi-algebraic cells, built inductively, and
whose projections onto the first $k$ variables, for $k < n$, are
either disjoint or the same. These cells are not just arbitrary
semi-algebraic sets homeomorphic to $(0,1)^d$, for some $d \in \NN$,
but cells that arise from graphs $\Gamma(g)$ of some semi-algebraic
functions~$g$, as below.

\begin{definition}\label{def:cad}
A \emph{cylindrical algebraic decomposition} or \emph{{c.a.d.}}\ of
$\RR^n$ is a finite semialgebraic cell decomposition $\cP=\cP_n$ of
$\RR^n$ defined inductively by the following conditions.
\begin{enumerate}
  \item If $n=1$ then $\cP=\cP_1$ is any finite cell decomposition of
    $\RR$.
  \item The projection $\pr_n\colon \RR^n\to \RR^{n-1}$ on the last
    $n-1$ coordinates is \emph{cylindrical}: that is, if
    $C,\,C'\in\cP_n$ then either $\pr_n (C)\cap \pr_n (C')=\emptyset$
    or $\pr_n (C) = \pr_n (C')$.
  \item $\cP_{n-1}=\{\pr_n (C)\mid C\in \cP_n\}$ is a c.a.d.\ of
    $\RR^{n-1}$.
  \item For each $D\in\cP_{n-1}$ there are finitely many continuous
    semi-algebraic functions $g_1,\ldots,g_k\colon D\to \RR$,
    satisfying $g_j(p)<g_{j+1}(p)$ for all $j$ and all $p\in D$, such
    that $\cP\ni \Gamma(g_j)=\{(p,y)\mid g_j(p)=y\}$ for each $j$ and
    $\cP\ni \Delta_j=\{(p,y)\mid g_j(p)<y<g_{j+1}(p,y)\}$:
    furthermore both $\Delta_{-\infty}=\{(p,y)\mid y<f_1(p)\}$ and
    $\Delta_\infty=\{(p,y)\mid f_k(p)<y\}$ also belong to $\cP$.
\end{enumerate}
\end{definition}

\begin{definition}\label{def:section}
In Definition~\ref{def:cad}, the graphs $\Gamma(g_j)$ are called
\emph{sections} and the cells $\Delta_j$ are called \emph{sectors} of
$\cP$.
\end{definition}

A c.a.d.\ is usually chosen to respect some data, such as some
functions on $\RR^n$ or subsets of $\RR^n$.

\begin{definition}\label{def:Finvariantcad}
Let $\bF\subset \RR[x_1,\dots,x_n]$ be a finite set of nonzero
polynomials. A c.a.d.\ $\cP$ is said to be \emph{$\bF$-invariant} if,
for every $f \in \bF$, the sign of $f$ is constant on each $C \in
\cP$.
\end{definition}

This is sometimes called \emph{sign-invariance}: one could instead
require other properties of $f$, such as its order of vanishing, to be
constant on each~$C$, but we shall not need any other kind of
invariance here.

\begin{definition}\label{def:adaptedcad}
Let $S\subset \RR^n$ be a semi-algebraic set. A c.a.d.\ $\cP$ of
$\RR^n$ is \emph{adapted to $S$} if $S$ is a union of cells of $\cP$.
\end{definition}
It is sometimes useful to give some more information about $\cP$.

\begin{definition}\label{def:sampledcad}
A \emph{sampled c.a.d.}\ is a c.a.d.\ $\cP$ together with a choice of
base point $b_C\in C$ for each cell $C\in \cP$.
\end{definition}
In general, a c.a.d.\ is not a CW complex: for instance, a c.a.d.\ can
fail to be closure finite.

\begin{definition}\label{def:strongcad}
We say that a c.a.d.\ $\cP$ of $\RR^n$ is a \emph{strong {c.a.d.}}\ if
$\cP$ is well-bordered and locally boundary connected.
\end{definition}
This is equivalent to the definition in \cite{lazard2010} in view of
Corollary~\ref{cor:wbimpliescf}.

\subsection{$\bF$-invariant {c.a.d.}s}\label{subsect:Finvt}

The aim of this section is to investigate when an $\bF$-invariant
c.a.d.\ adapted to a closed bounded semi-algebraic set $S$ exhibits
$S$ as a regular cell complex.

\begin{definition}\label{def:reduced}
If $\cP$ is an $\bF$-invariant c.a.d.\ and $C\in\cP$ is a section, we
put $\bF_C=\{f\in \bF\mid f|_C\equiv 0\}$. We say that $(\cP,\bF)$
is \emph{reduced} if $\bF_C\neq \emptyset$, for every section
$C\in \cP$.
\end{definition}

We do not require that $C$ should actually be cut out by $\bF_C$, but
we will usually impose the next condition, which may be seen as a
weaker version.

\begin{definition}\label{def:wellbased}
If $\cP$ is an $\bF$-invariant {c.a.d.}, we say that a section
$C\in\cP$, with $\pr_n(C)=D\in\cP_{n-1}$, is a \emph{bad cell} for
$(\cP,\bF)$ if there is an $f\in\bF_C$ such that
$f|_{\pr_n^{-1}(D)}\equiv 0$. If there are no bad cells, we say that
$(\cP,\bF)$ is \emph{well-based}.
\end{definition}

Note that the definitions of bad cell and well-based depend on $\bF$
as well as $\cP$. Note also that we apply the term ``bad cell'' to
$C$, not $D$. This makes no difference, since if $C$ is a bad cell
then so is any $C'$ with $\pr_n(C')=\pr_n(C)$ (consider the same $f$),
but it does reflect our point of view of starting with a given c.a.d.\
rather than constructing one inductively.

We aim to show that reduced well-based strong c.a.d.s give regular
cell complexes (see Theorem~\ref{thm:wellbasedstronggivesreg} for the
precise statement). This was shown in \cite[Theorem~2]{schwartz1983}
for a c.a.d.\ $\cP$ constructed via Collins' algorithm
\cite{collins}: by \cite[Theorem~4.4]{lazard2010}, such a $\cP$ is
strong if it is well-based. Other algorithms are now in use, though,
such as c.a.d.\ via regular chains~\cite{ChenMorenoMaza2014b} or via
comprehensive Gr\"obner systems~\cite{Fukasakuetal2015a},
so we want to be able to dispense with the condition on the
construction.

We need two lemmas: the first is a variant of \cite[Lemma~2.5.6]{BCR}.

\begin{lemma}\label{lem:extendwellbased}
Suppose that $\bF\subset\RR[x_1,\dots,x_{n-1},y]$ is a finite set of
non-zero polynomials and denote by
\[
\bF'=\left\{\frac{\partial^r\! f}{\partial y^r}  \mid
  r\in\ZZ_{\ge 0},\,f\in\bF,\,\frac{\partial^r\! f}{\partial
  y^r}\not\equiv 0\right\} 
\]
its closure under the operator $\frac{\partial}{\partial y}$. Let
$C=\Gamma(g)$ be a bounded section in an $\bF'$-invariant c.a.d.\
$\cP$ of $\RR^n$, for a semi-algebraic continuous bounded function
$g \colon D = \pr_n(C) \to \RR$. If $\bF_C\neq\emptyset$ and
$(\cP,\bF)$ is well-based, then $g$ can be extended continuously to
$\Bar{D}$.
\end{lemma}
\begin{proof}
The proof is similar to step~(ii) in the proof of
\cite[Lemma~2.5.6]{BCR}. It is enough to show that $g$ extends
continuously to $D\cup\{p\}$, for an arbitrary $p\in \Bar D$.

By the Curve Selection Lemma \cite[Theorem~2.5.5]{BCR}, we choose a
continuous semi-algebraic path $\eta\colon [0,1]\to \Bar
D\cap\BB(p,1)$, such that $\eta(0)=p$ and $\eta(t)\in D$ for
$t>0$. Then we define $\Tilde\eta(t)=g(\eta(t))\in \RR$, for $t>0$:
since $g$ is bounded by hypothesis, $\Tilde\eta\colon (0,1]\to \RR$ is
a bounded continuous semialgebraic function and hence extends
continuously to $\Tilde\eta\colon [0,1]\to \RR$ by
\cite[Proposition~2.5.3]{BCR}.

Now we extend $g$ to $\Bar g\colon D\cup\{p\}\to \RR$ by putting $\Bar
g(p)=\Tilde\eta(0)$, and $\Bar g=g$ on $D$. The claim is that $\Bar g$
is continuous at~$p$. If not, then
\[
\exists\, \varepsilon>0 \,\, \forall\delta >0\,\, \exists\, q\in
D\ \text{such that}\ 
 \|q-p\|<\delta\ \text{and}\ |g(q)-\Bar g(p)|\ge \varepsilon,
\]
and hence if we define $E=\{q\in D\mid |g(q)-\Bar g(p)|\ge
\varepsilon\}$ then $p\in \Bar E$. Again applying the Curve Selection
Lemma we obtain a path $\theta\colon [0,1]\to \Bar E$ with
$\theta(0)=p$ and $\theta(t)\in E$ for $t>0$, so exactly as before
we put $\Tilde{\theta}(t)=g(\theta(t))$ and this extends
continuously to $\Tilde{\theta}\colon [0,1]\to \RR$. By continuity,
we have $|\Tilde\eta(0)-\Tilde{\theta}(0)|\ge \varepsilon$, and
also $(p,\Tilde\eta(0))\in \Bar C$ and $(p,\Tilde{\theta}(0))\in
\Bar C$.

Now suppose that $f\in \bF_C$, so $f\in \bF$ and $f|_C\equiv
0$. Consider the polynomial $f_p(y)=f(p,y)\in \RR[y]$, and observe
that $f_p(\Tilde\eta(0))=f_p(\Tilde{\theta}(0))=0$. If $f_p$ is not
the zero polynomial, we may consider the set $\{f_p, \frac{df_p}{dy},
\frac{d^2\!f_p}{dy^2},\ldots\}$ of all derivatives of $f_p$. By
$\bF'$-invariance, for any given $r$ the sign of
$\frac{\partial^r\! f}{\partial y}$ is the same near
$(p,\Tilde\eta(0))$ (say at $(\eta(t),\Tilde\eta(t))$ for small
$t>0$) as near $(p,\Tilde\theta(0)$. Hence $\frac{d^r\! f_p}{dy^r}$
cannot have opposite signs at $y=\Tilde\eta(0)$ and
$y=\Tilde\theta(0)$, although one might be zero and the other not.

But this contradicts Thom's Lemma \cite[Proposition~2.5.4]{BCR}: at
two distinct zeros of a real polynomial in one variable, some
derivative must have opposite signs.

Hence, if $\Bar g$ is discontinuous at~$p$, we must have $f_p\equiv
0$: that is, $f$ is identically zero above~$p$. But then, by
cylindricity and $\bF'$-invariance, $f$ must be identically zero above
the cell in $\cP_{n-1}$ containing~$p$, contrary to the assumptions.
\end{proof}

Next we need a lemma that allows us to pass extensions up from subdivisions.

\begin{lemma}\label{lem:extendbdysmooth}
Let $C=\Gamma(g)$ be a bounded, local boundary connected section in an
$\bF$-invariant c.a.d.\ $\cP$ of $\RR^n$, for a semi-algebraic
continuous function $g\colon D =\pr_n(C) \to \RR$. Suppose that $\cP'$
is a well-based strong $\bF$-invariant c.a.d.\ refining $\cP$ (i.e.\
each cell in $\cP'$ is a subset of a cell of $\cP$), in which $C$ is
partitioned into sections $C_i=\Gamma(g_i)$ for semi-algebraic
continuous functions $g_i \colon D_i=\pr_n(C_i) \to \RR$. If all the
$g_i$ extend continuously to $\Bar{D}_i$, then $g$ extends
continuously to $\Bar{D}$.
\end{lemma}

\begin{proof}
It is enough to show that if $p \in \partial D \cap \partial D_i
\cap\partial D_j$ for some $i \neq j$, then $\Bar{g}_i(p) =
\Bar{g}_j(p)$.  Then $\Bar{g} \colon \Bar{D} \to \RR$ is consistently
defined by $\Bar{g}(p) = \Bar{g}_i(p)$ if $p \in \Bar{D}_i$.

We first show this with the assumption that $D_j\preceq D_i$. Then,
since $\cP'$ and therefore $\cP'_{n-1}$ are strong and in particular
closure finite, we have $D_j\subseteq\Bar{D_i}$ by
Lemma~\ref{lem:cfbdysubadj}. Therefore $g_j$ agrees with $\Bar{g}_i$
on $D_j$ (they both agree with $g$) and hence $\Bar{g}_j$ agrees with
$\Bar{g}_i$ also on $\Bar{D_j}$, which is contained in $\Bar{D_i}$.

For general $i$ and $j$, we construct a subadjacency chain from $D_i$
to $D_j$, by considering a semi-algebraic path between $D_i\cap
B(p,\epsilon)$ and $D_j\cap B(p,\epsilon)$; this is possible as
$D \cap B(p,\epsilon)$ is semi-algebraic and connected, and thus
semi-algebraically path-connected by
\cite[Prop.\ 2.5.13]{BCR}. The path gives us a way of selecting the
correct consecutive cells.

Let $\eta \colon [0,1] \to C\cap B(p,\epsilon)$ be a semi-algebraic
path with $\eta(0)\in D_i\cap B(p,\epsilon)$ and $\eta(1)\in D_j\cap
B(p,\epsilon)$. As $\eta([0,1])$ semi-algebraic, $\eta([0,1])) \cap
D_k$ has finitely many connected components, for any $D_k$; thus
$\eta^{-1}(D_k)$ is a finite collection of intervals contained in
$[0,1]$.

Considering the preimage of every cell in $\cP'_{n-1}$, we get a
finite partition $[0,1] = \bigcup_{\ell=1}^N I_\ell$ where $\sup
I_\ell = \inf I_{\ell+1}=t_\ell$. Denote by $D_{k(\ell)}$ the unique
cell that contains $\eta(I_\ell)$. Then
$\eta(t_\ell)\in\Bar{D}_{k(\ell)}\cap\Bar{D}_{k(\ell+1)}$, so
$D_{k(\ell)}$ and $D_{k(\ell+1)}$ are adjacent: moreover
$\eta(t_\ell)$ belongs to either $D_{k(\ell)}$ or $D_{k(\ell+1)}$, so
one is subadjacent to the other.

Now we have a finite chain of not necessarily distinct cells
\[
D_i=D_{k(1)}\diamondsuit D_{k(2)}\diamondsuit\dots\diamondsuit
D_{k(N)}=D_j,
\]
where each $\diamondsuit$ stands for either $\preceq$ or
$\succeq$. Hence $\Bar{g}_i(p)=\Bar{g}_{k(2)}(p)=\dots=\Bar{g}_j(p)$,
so $\Bar{g}(p)$ is well-defined.
\end{proof}

\begin{theorem}\label{thm:wellbasedstronggivesreg}
Suppose $\cP$ is an $\bF$-invariant, reduced, well-based strong
c.a.d.\ of $\RR^n$ adapted to a closed and bounded subset $S$. Then
the corresponding decomposition of $S$ is a regular cell complex.
\end{theorem}

\begin{proof}
A strong c.a.d.\ $\cP$ of $\RR^n$ is closure finite, so we just need
to show that $C \subset S$ is a regular cell. Moreover, if we can show
that the sections of $\cP$ are regular, then by
\cite[Lemma~5]{schwartz1983}, so are the sectors.

We will show that a section $C=\Gamma(g)$ is a regular cell by proving
that $(\Bar{C},C)$ is homeomorphic to $(\Bar{D},D)$, where $D
=\pr_n(C)\in \cP_{n-1}$ is the cell below $C$; then the result follows
by induction. It suffices to show that $g\colon D\to \RR$ extends
continuously to $\Bar{D}$, since then $\Bar g\colon (\Bar D, D)\to
(\Bar C, C)$ and $\pr_n\colon (\Bar C, C) \to (\Bar D, D)$ are
mutually inverse homeomorphisms.

Let $\bF'$ be the closure of $\bF$ under $\frac{\partial}{\partial
x_n}$; that is, the smallest set that contains $\bF$ and is closed
under partial differentiation with respect to $x_{n}$. We may choose a
$\bF'$-invariant c.a.d.\ $\cP'$ refining $\cP$. Then $C$ is
partitioned into sections $C_i\in \cP'$ and each $C_i=\Gamma(g_i)$ for
the continuous semi-algebraic function $g_i=g|_{D_i}$ (with, as usual,
$D_i=\pr_n(C_i)$).

Now we apply Lemma~\ref{lem:extendwellbased}, remembering that since
$C_i\subseteq C$ there is an $f\in \bF$ that vanishes on $C_i$. We
conclude that each $g_i$ can be extended continuously to $\Bar{D}_i$.

Finally, we use Lemma~\ref{lem:extendbdysmooth} to extend $g$
continuously to $\Bar{D}$.
\end{proof}

The above, in combination with Lemma~\ref{lem:regthencf=wb}, prompts
us to raise the following question.

\begin{question}\label{qu:cadthencf=wb}
Suppose that $\cP$ is a c.a.d.\ of $\RR^n$: is it true that $\cP$ is
closure finite if and only if it is well-bordered?
\end{question}

Some of the results in this section can be strengthened slightly, by
weakening global conditions so that they only apply where they are
needed. For example, in Theorem~\ref{thm:wellbasedstronggivesreg} it
would be enough for $\cP$ to give an $\cF$-invariant c.a.d.\ of some
open cylinder containing $S$ rather than of the whole of
$\RR^n$. Similarly, in Lemma~\ref{lem:extendwellbased} it is enough
for $(\cP,\bF_C)$ to be well-based (even just near $C$), and the
condition of $\bF'$-invariance can also be relaxed analogously.

A related question is whether (or how far) the results of this section
apply to c.a.d.s invariant for a formula, as in~\cite{McCallum1999a}
and~\cite{Bradfordetal2016a}.

\subsection{Topology of strong {c.a.d.}s}\label{subsect:topology}

The following conjecture is made in \cite[p.~94]{lazard2010}.

\begin{conjecture}\label{conj:strongcadisreg}
Suppose $\cP$ is a strong c.a.d.\ of $\RR^n$ adapted to a closed and
bounded semi-algebraic set $S$. Then the corresponding decomposition
of $S$ is a regular cell complex.
\end{conjecture}

We prove some cases of this using techniques from topology. By the
nature of the proofs, they are valid only for $\RR$, not for arbitrary
real closed fields. The idea is that a regular cell is automatically a
manifold with boundary: conversely, it is sometimes possible to give
conditions on a manifold with boundary that are sufficient to ensure
that it is a regular cell, and in low dimension we are able to verify
that these conditions always hold.

\begin{definition}
\label{def:homology/homotopysphere}
Let $X$ be a compact $d$-manifold. We say that $X$ is a \emph{homology
$d$-sphere} if $X$ has the same homology groups as $\SS^d$. We say
that $X$ is a \emph{homotopy $d$-sphere} if $X$ has the same homotopy
type as $\SS^d$.
\end{definition}

It follows from the Hurewicz theorem \cite[Thm~4.32]{hatcher2002} and
Whitehead's theorem \cite[Thm~4.5]{hatcher2002} that any simply
connected homology sphere $Y$ is a homotopy sphere. One must heed the
warning given in \cite{hatcher2002} after the proof of Whitehead's
theorem: we need a weak homotopy equivalence, that is, a map that
induces isomorphisms between the homotopy groups. However, if
$B\subset Y$ is an open ball then a map that identifies $B\imic\BB^d$
with the complement of a point $q\in\SS^d$ and sends all of
$Y\setminus B$ to $q$ is a weak homotopy equivalence between $Y$ and
$\SS^d$.

\begin{definition}\label{def:mfldwbdy}
A topological space $X$ is a \emph{$n$-dimensional manifold with
boundary} if, for each $x \in X$, there exists a neighbourhood $V$ of
$x$ that is homeomorphic to an open set in either $\RR^n$ or
$(\RR_{\ge 0})^n$.  The \emph{manifold boundary of $X$}, denoted
$\partial_m X$, is the set of points of $X$ with no neighbourhood
homeomorphic to an open set in $\RR^n$.

If $X=\Bar C$ for some $C\subset\RR^m$ we say that $\Bar{C}$ is
\emph{compatibly a manifold with boundary} if
$\partial_m\Bar{C}=\partial_cC$.
\end{definition}

For more on manifolds with boundary see~\cite[p~252]{hatcher2002}. Our
main use of them is based on the following easy result.

\begin{lemma}\label{lem:bdyishomologysphere}
Let $X$ be a $d$-dimensional compact, contractible manifold with
boundary. Then the boundary $\partial_m X$ of $X$ is a homology
$(d-1)$-sphere.
\end{lemma}

\begin{proof}
The homology sequence of $(X,\partial_m X)$ is
\[
\dotsb \to H_i(\partial_m X) \to H_i(X)\to H_i(X,\partial_m X) \to
H_{i-1}(\partial_m X) \to \dotsb
\]
and Lefschetz duality \cite[Theorem 3.43]{hatcher2002} gives
$H_k(X,\partial_m X) \cong H^{d-k}(X)$. Since $X$ is contractible we
have $H_0(X)\imic H^0(X)\imic\ZZ$ and $H_i(X)=H^i(X)=0$ for $i\neq 0$,
and hence $H_0(\partial_m X)\imic H_n(\partial_m X)\imic \ZZ$ and
$H_i(\partial_m X)=0$ for $i\neq 0,\,n$.
\end{proof}

In order to apply this to c.a.d.s we need a result
(Theorem~\ref{thm:closedcadcellcontractible}) on the contractibility
of $\Bar{C}$, for a c.a.d.\ cell~$C$. First we recall the following
theorem of Smale \cite{smale1957}.

\begin{theorem}\label{thm:smale}
Suppose that $f\colon X\to Y$ is a proper surjective continuous map
between connected, locally compact separable metric spaces, and $X$ is
locally contractible. If all the fibres of $f$ are contractible and
locally contractible, then $f$ is a weak homotopy equivalence.
\end{theorem}

We use Theorem~\ref{thm:smale} to deduce the following, which may be
of independent interest.

\begin{theorem}\label{thm:closedcadcellcontractible}
Suppose $\cP$ is a c.a.d.\ of $\RR^n$ and the induced c.a.d.\
$\cP_{n-1}$ of $\RR^{n-1}$ is strong. Then the closure $\Bar{C}$ of
any bounded cell $C$ of $\cP$ is contractible.
\end{theorem}

\begin{proof}
As often in Section~\ref{subsect:Finvt}, let $D=\pr_n(C)\in\cP_{n-1}$.
We shall show that $\Bar{C}$ and $\Bar{D}$ have the same homotopy
type: then $\Bar{C}$ is contractible by induction on $n$. The result
is true for $n=1$, so we assume $n>1$.

As $C$ and $D$ are bounded, $\Bar{C}$ and $\Bar{D}$ are compact
semi-algebraic sets and thus, by \cite[Thm~9.4.1]{BCR}, $\Bar{C}$ and
$\Bar{D}$ each admits a CW-complex structure. By Whitehead's theorem,
it suffices to show that $\Bar{C}$ and $\Bar{D}$ are weakly homotopy
equivalent, and for that it is enough to verify that
$\pr_n|_{\Bar{C}}$ satisfies the conditions of
Theorem~\ref{thm:smale}.

The spaces $\Bar{C}$ and $\Bar{D}$ are connected as they are the
closures of connected spaces. Moreover, as they are also Hausdorff
compact metric spaces, they are locally compact and separable. It
follows immediately from the local conic structure theorem
\cite[Thm~9.3.6]{BCR} that any semi-algebraic set, in
particular $\Bar{C}$ and every fibre of $\pr_n|_{\Bar{C}}$, is locally
contractible. The map $\pr_n$ is a continuous map between a compact
space and a Hausdorff space, so it is closed and proper.

Finally, because $\cP_{n-1}$ is a strong {c.a.d.}, by
\cite[Proposition~5.2]{lazard2010} the fibres of $\pr_n|_{\Bar{C}}$
are closed segments and thus contractible.
\end{proof}

Note that in Theorem~\ref{thm:closedcadcellcontractible} we do not
need the c.a.d.\ $\cP$ to be strong. We do need $\cP_{n-1}$ to be
strong in order to apply \cite[Proposition~5.2]{lazard2010}.

\begin{question}\label{qu:cadcellcontractible}
What conditions on a c.a.d.\ $\cP$ are necessary to ensure that the
closures of its bounded cells are contractible? Could it be true for
an arbitrary~$\cP$?
\end{question}

In fact a compact contractible manifold with boundary $X$ is a
$d$-cell ($d\ge 3$) if and only if $\partial_m X$ is simply
connected. For this one uses the topological $h$-cobordism theorem,
which is a consequence of the (generalised) Poincar\'e conjecture: see
\cite{schultz2007} for some comments on this, and also
\cite[Conjecture~3.5$\vphantom{5}'$]{hempel2004}.

More precisely: $X$ is contractible so $\partial_m X$ is a homology
$(d-1)$-sphere by Lemma~\ref{lem:bdyishomologysphere}, so $\partial_m
X$ is a homotopy $(d-1)$-sphere because it is simply-connected. Then
$\partial_m X$ is cobordant with the sphere $\SS_0$ bounding a small
ball $\BB_0\subset X\setminus \partial_m X$. Because this cobordism is
an $h$-cobordism it is homeomorphic to $\SS_0\times [0,1]$, and we get
that $X$ is homeomorphic to $\BB_0\cup_{\SS_0}\SS_0\times[0,1]$ which
is $\Bar\BB^d$.

Consequently, to show that a c.a.d.\ $\cP$ of a compact
$d$-dimensional semi-algebraic set is a regular cell complex it is
enough to show that every cell closure $\Bar{C}$ is compatibly a
compact contractible manifold with boundary and, for $d> 3$, that
$\partial_c C$ is simply connected.

In general, the closure of a c.a.d.\ cell is not compatibly a manifold
with boundary. The cell $W$ in Example~\ref{ex:whitney} is an example
of this: $\Bar{W}\imic \Bar\BB^2$ but $\partial_m\Bar{W}$ is strictly
contained in $\partial_c W$. Another example is the non-regular cell
$\BB^2\setminus \{(0,y)\mid y\ge 0\}$ mentioned after
Definition~\ref{def:regularcell}.

In very low dimension the position is simple.

\begin{lemma}\label{lem:01lbcgivesmfldwbdy}
If $C$ is a locally boundary connected $0$- or $1$-cell of $\RR^n$,
then $\Bar{C}$ is compatibly a manifold with boundary.
\end{lemma}

This is immediate from the definition of manifold with boundary.
Using a more involved argument, we can show that, under certain
conditions, closures of $2$-cells are compatibly manifolds with
boundary.

\begin{definition}\label{def:homologymfld}
A locally compact space $X$ is a \emph{homology $n$-manifold} if, for
all $p \in X$
\[
H_i(X,X\setminus \{p\}) =
\begin{cases}
  \ZZ & \mbox{ if } i = n,\\
 0 & \mbox{ otherwise.}
\end{cases}
\]
\end{definition}

A homology manifold is not a manifold in general, but homology
$n$-manifolds are $n$-manifolds for $n\le 2$. In some ways, they
behave better than manifolds, as the following fact (stated
in~\cite{hazewinkel2013}, and easily checked by the K\"unneth formula)
suggests.

\begin{lemma}\label{lem:factorsofmfldarehmflds}
Suppose $X$ and $Y$ are topological spaces. If $X\times Y$ is a
topological manifold, then $X$ and $Y$ are homology manifolds.
\end{lemma}

If $C$ is a semi-algebraic cell and $p\in\partial C$ then
$\BB(p,\varepsilon)\cap C$ is a open subset of a manifold and thus a
manifold. As in the proof of Proposition~\ref{prop:BretractsS}, the
local conic structure gives a homeomorphism
\[
\BB(p,\varepsilon)\cap C \to (\SS(p,\varepsilon)\cap C)\times (0,1),
\]
so $\SS(p,\varepsilon)\cap C$ is a homology manifold for
$0<\varepsilon\ll 1$.

We can determine when locally boundary connected $2$-cells of a cell
decomposition are manifolds with boundary.

\begin{proposition}\label{prop:lbc2cell+givesmfldwbdy}
Let $C$ be a locally boundary connected semi-algebraic $2$-cell in
$\RR^n$. If, for all $p \in \partial_c C$ and $0<\varepsilon\ll 1$,
the set $\SS(p,\varepsilon)\cap C$ is homeomorphic to $(0,1)$, then
$\Bar{C}$ is compatibly a manifold with boundary.
\end{proposition}

\begin{proof}
A semi-algebraic homeomorphism $\gamma\colon(0,1)\to
\SS(p,\varepsilon)\cap C$ extends continuously to a map
$\Bar{\gamma}\colon [0,1]\to \Bar{C}$ by \cite[Proposition~2.5.3]{BCR}.  If
$\Bar\gamma(0)=\Bar\gamma(1)$ (for $\varepsilon\ll 1$) then
$\SS(p,\varepsilon)\cap C$ is not locally boundary connected, which by
Corollary~\ref{cor:homotopylb} would contradict the assumption on~$C$.

Therefore $\Bar{\gamma}(0)\neq \Bar{\gamma}(1)$, giving a
homeomorphism $\Bar{\gamma} \colon ([0,1],(0,1))\to
\SS(p,\varepsilon)\cap (\Bar{C},C)$.

Hence by the local conic structure of semi-algebraic sets,
$\Bar\BB(p,\varepsilon) \cap (\Bar{C},C)$ is homeomorphic to the cone
$K$ on $([0,1],(0,1))$, by a map sending $p$ to the vertex. This is a
manifold with boundary, and $p\in\partial_mK$, so $\Bar{C}$ is a
manifold with boundary and $\partial_cC\subseteq \partial_m\Bar{C}$.
But $C\cap\partial_m\Bar{C}=\emptyset$ because $C$ is a manifold, so
$\Bar{C}$ is compatibly a manifold with boundary.
\end{proof}

We can apply this to strong c.a.d.s.

\begin{lemma}\label{lem:strong2cellgivesmfldwbdy}
Let $C$ be a $2$-cell of a strong c.a.d.\ of $\RR^n$. Then $\Bar{C}$
is compatibly a manifold with boundary.
\end{lemma}

\begin{proof}
By the previous discussion, $\SS(p,\varepsilon)\cap C$ is a connected
$1$-manifold so it is homeomorphic to either $\SS^1$ or $(0,1)$.  As a
strong c.a.d.\ is well-bordered, there exists a $1$-cell $C' \subset
\partial C$, such that $p \in \Bar{C'}$. If $\SS(p,\varepsilon)\cap C$
is homeomorphic to $\SS^1$, then $\BB(p,\varepsilon)\cap \Bar{C}$ is
the cone with vertex $p$ on $\SS^1$, but that has an isolated boundary
point at $p$.
\end{proof}

We have proved the following result.

\begin{corollary}\label{cor:2strongcadreg}
Let $S \subset \RR^n$ be a $2$-dimensional compact semi-algebraic
set. If $\cP$ is a strong c.a.d.\ adapted to $S$, then $\cP$
represents $S$ as a regular cell complex.
\end{corollary}

To prove Conjecture~\ref{conj:strongcadisreg} for $n=3$ by this
method, we would need to show that $3$-cells of a strong c.a.d.\ of
$\RR^3$ are manifolds with boundary. We have not been able to do this
but we can do so under the additional assumption that the cells of the
c.a.d.\ are locally boundary simply connected: see
Definition~\ref{def:locallyboundary}.

\begin{theorem}\label{thm:vstrongcadismfldwbdy}
Let $\cP$ be a strong c.a.d.\ of $\RR^3$. If $C$ is a locally boundary
simply connected $3$-cell of $\cP$, then $\Bar{C}$ is compatibly a
manifold with boundary.
\end{theorem}
\begin{proof}
For $p \in \partial C$ and $0<\varepsilon \ll 1$, we know that
$\SS(p,\varepsilon)\cap C$ is a homology $2$-manifold and therefore a
$2$-manifold.

We claim that $\SS(p,\varepsilon)\cap C$ is a $2$-cell that satisfies
the conditions of Proposition~\ref{prop:lbc2cell+givesmfldwbdy}. Then,
as before, $\SS(p,\varepsilon)\cap C$ is a regular cell and the
interior of the cone on its closure is homeomorphic to $[0,1)^3$,
which is compatibly a manifold with boundary.

As in the proof of Proposition~\ref{prop:lbc2cell+givesmfldwbdy},
$\SS(p,\varepsilon)\cap C$ is locally boundary connected because of
Corollary~\ref{cor:homotopylb}.

If $q \in \partial_c(\SS(p,\varepsilon)\cap C)$, then
$(\SS(p,\varepsilon)\cap C)\cap \SS(q,\varepsilon')$ is a connected
manifold. If it is homeomorphic to $\SS^1$ then
$(\SS(p,\varepsilon)\cap C)$ is not locally boundary $1$-connected so
$C$ is also not locally boundary $1$-connected, again by
Corollary~\ref{cor:homotopylb}.

Lastly, we prove that $\SS(p,\varepsilon)\cap C$ is a cell: as the
dimension is~$2$, it is enough to show that it is contractible.  It is
a metrisable manifold, so by \cite[Corollary~1]{milnor1959} it has the
homotopy type of a CW complex.  Thus, by Whitehead's theorem, it
suffices to show that $\SS(p,\varepsilon)\cap C$ has trivial homotopy
groups: $\pi_0$ and $\pi_1$ are trivial by assumption. By Hurewicz's
theorem it is enough to show that all the homology groups
vanish. Certainly $H_i(\SS(p,\varepsilon)\cap C)=0$ for all $i\geq 3$,
as $\SS(p,\varepsilon)\cap C$ is a $2$-manifold.

But $\SS(p,\varepsilon)\cap C$ is a non-compact connected
$2$-manifold, otherwise $p$ is an isolated point of $\partial_c C$
exactly as in Lemma~\ref{lem:strong2cellgivesmfldwbdy}, so
$H_2(\SS(p,\varepsilon)\cap C) = 0$ by \cite[Cor~VIII.3.4]{dold1995}.
\end{proof}

\begin{theorem}\label{thm:vstrongcadisreg}
Suppose that $S \subset \RR^3$ is semi-algebraic and $\cP$ is a strong
c.a.d.\ adapted to $S$, such that every $3$-cell of $\cP$ is locally
boundary simply connected. Then $\cP$ yields a regular cell complex of
$S$.
\end{theorem}

\begin{proof}
This follows from Theorem~\ref{thm:vstrongcadismfldwbdy} and the
discussion after Question~\ref{qu:cadcellcontractible}.
\end{proof}

In the light of the above arguments, we consider it likely that
Conjecture~\ref{conj:strongcadisreg} in full generality requires that
the cells of the c.a.d.\ should be assumed to be locally boundary
contractible, not just locally boundary connected.

Extending our approach, even with that stronger hypothesis, to higher
dimension, we encounter at least two difficulties. First, if
$[0,1]\times M$ is a $3$-manifold then $M$ is a $2$-manifold, and we
used this in the proof of Theorem~\ref{thm:vstrongcadismfldwbdy}, but
the corresponding assertion in higher dimension is false. Second, we
also used the fact that a contractible $2$-manifold is a cell, which
also fails in higher dimension: it would be enough to show that in
addition the boundary is a sphere, the main obstacle being to
determine whether the boundary is simply connected.

We would encounter these difficulties, for instance, if we tried to
prove a version of Proposition~\ref{prop:lbc2cell+givesmfldwbdy}
(e.g.\ with $C$ being locally boundary contractible and
$\SS(p,\varepsilon) \cap C$ a $(d-1)$-cell), but we could also try to
exploit the fact that $C$ is a c.a.d.\ cell, rather than just
semi-algebraic.

\section{Subadjacency and order complex}\label{sect:ordercomplex}

In \cite{lazard2010}, Lazard poses the following question (as
rephrased, but not altered, by us).

\begin{question}\label{qu:lazardunrevised}
Let $\cP$ be a strong sampled c.a.d.\ adapted to a compact
semi-algebraic set $S$ and let $\cE$ be the set of cells of $\cP$
contained in $S$. For each subadjacency chain $\bE=(E_0\prec
E_1\prec\dots\prec E_k)$ with $E_j\in\cE$, we let $\sigma_\bE$ be the
convex hull of the sample points $b_{E_0},\dots,b_{E_k}$.
\begin{enumerate}
\item[(i)] Is $\{\sigma_\bE\}$ a simplicial complex?
\item[(ii)] Is $\bigcup_\bE\sigma_\bE$ homeomorphic to $S$?
\end{enumerate}
\end{question}
Some general position condition on the sample points is needed,
otherwise $\sigma_\bE$ could even fail to be a $k$-simplex. Even so,
the answer to (i) as posed above is no, because there may be
intersections, as the example below illustrates.
\begin{center}
\begin{tikzpicture}[scale=2]
\draw (0,0) to [out=70, in=180] (0.3,0.3) to [out=0, in=110] (0.6, -0.5)
to [out=290, in=220](1,0) to [out=265, in=330] (0.5, -1) to [out=150,
  in=300] (0,0);
\draw[green] (0.9, -0.4)--(0,0)--(0.6,-0.5)--cycle;
\draw[green] (0.9, -0.4)--(1,0)--(0.6,-0.5)--cycle;
\draw[green] (0.9, -0.4)--(0,0)--(0.5,-1)--cycle;
\draw[green] (0.9, -0.4)--(1,0)--(0.5,-1)--cycle;
\draw[red, fill=red](0,0) circle [radius = 0.025];
\draw[red, fill=red](1,0) circle [radius = 0.025];
\draw[red, fill=red](0.9,-0.4) circle [radius = 0.025];
\draw[red, fill=red](0.6,-0.5) circle [radius = 0.025];
\draw[red, fill=red](0.5,-1) circle [radius = 0.025];
\node[below] at (0.5,-1.3){\rm Question~\ref{qu:lazardunrevised}(i):
  not a simplicial complex};
\end{tikzpicture}
\end{center}
However, this is easily corrected: instead of working inside $\RR^n$
one should replace $\{\sigma_\bE\}$ by the order complex $\Delta(\cE)$
of the poset $(\cE,\preceq)$ (it follows immediately from
Lemma~\ref{lem:cfbdysubadj} that $\preceq$ is transitive) and
$\bigcup_\bE\sigma_\bE$ by the geometric realisation
$\|\Delta(\cE)\|$. See \cite{bjorner1996} for details of order
complexes and barycentric subdivision.

\begin{question}\label{qu:lazardrevised}
Let $\cP$ be a strong c.a.d.\ adapted to a compact semi-algebraic set
$S$ and let $\cE$ be the set of cells of $\cP$ contained in $S$. Is
$\|\Delta(\cE)\|$ homeomorphic to $S$, i.e.\ is $\Delta(\cE)$ a
triangulation of~$S$?
\end{question}

A regular cell complex can be thought of as a CW-complex that is one
barycentric subdivision from being a triangulation.

\begin{lemma}\label{lem:baryofregissimplicial}
Let $\Sigma$ be a regular cell complex and let $\Sigma^*$ be the set
of all closed cells ordered by inclusion. Then $\| \Sigma \| \cong
\|\Delta(\Sigma^*) \|$.
\end{lemma}
For details, see \cite[12.4\ (ii)]{bjorner1996}. It is easy to see
that in general a regular cell complex is not a triangulation
(e.g.\ \cite[Example~5.4]{basuroypollack}).

\begin{theorem}\label{thm:strongcadposet}
Let $\cP$ be a strong c.a.d.\ of $\RR^n$. The partially ordered set
$(\cP,\preceq)$ of cells with respect to sub-adjacency is isomorphic
to the partially ordered set $(\cP^*,\subseteq)$ of closed cells with
respect to inclusion.
\end{theorem}
\begin{proof}
The bijection $\cP\to\cP^*$ is given by closure, $C\mapsto
\Bar{C}$. The inverse is given by relative interior, taking a closed
cell $Z$ to its interior in its Zariski closure, or by taking $Z$ to
the cell (unique, by Lemma~\ref{lem:cfbdysubadj}) of dimension equal
to $\dim Z$ that meets $Z$. We need to show that $C \preceq D$ if and
only if $\Bar{C} \subseteq \Bar{D}$, which follows immediately from
Lemma~\ref{lem:cfbdysubadj}.
\end{proof}

Thus Conjecture~\ref{conj:strongcadisreg} would imply an affirmative
to Question~\ref{qu:lazardrevised}. Note also that in
Theorem~\ref{thm:strongcadposet} the cylindricity is not used, so
$\cP$ does not need to be a {c.a.d.}, only a semi-algebraic cell
decomposition.

\end{document}